\documentclass[12pt]{article}
\usepackage{amssymb,amsmath,amsthm,amsfonts,mathtools}
\usepackage{url}
\usepackage{hyperref}
\usepackage[shortlabels]{enumitem}
\usepackage[square,numbers]{natbib}

\newtheorem{theorem}{Theorem}[section]
\newtheorem{dfn}[theorem]{Definition}
\newtheorem{conj}{Conjecture}
\newtheorem{lem}[theorem]{Lemma}
\newtheorem{prop}[theorem]{Proposition}
\newtheorem*{remark}{Remark}

\theoremstyle{definition}
\newtheorem{exa}{Example}

\begin{document}

\title{Reed-Solomon codes over small fields with constrained generator matrices}

\author{Gary Greaves\thanks{G.G. was supported by  the Singapore
Ministry of Education Academic Research Fund (Tier 1);  grant number:  RG127/16.}  { }and Jeven Syatriadi, \\
  School of Physical and Mathematical Sciences, \\
  Nanyang Technological University, \\
   21 Nanyang Link, Singapore 637371\\
 \tt{\{gary,jsyatriadi\}@ntu.edu.sg}
}
\date{}

\maketitle

\begin{abstract}
  We give constructions of some special cases of $[n,k]$ Reed-Solomon codes over finite fields of size at least $n$ and $n+1$ whose generator matrices have constrained support.
  Furthermore, we consider a generalisation of the GM-MDS conjecture proposed by Lovett in 2018.
  We show that Lovett's conjecture is false in general and we specify when the conjecture is true.
\end{abstract}

\section{Introduction}

A linear code of length $n$, dimension $k$, distance $n-k+1$, and alphabet size $q$ is called an \emph{$[n,k]_q$ MDS code}~\cite{book}.
For certain applications (see \cite{dau2013balanced},\cite{dau2014existence},\cite{dau2015simple},\cite{halbawi2014distributed},\cite{yan2014weakly}), one would like to construct an $[n,k]_q$ MDS code having:
\begin{enumerate}
  \item a generator matrix with prescribed zero pattern;
  \item a small alphabet size (i.e., $q$ close to $n$).
\end{enumerate}

It was conjectured by Dau et al.~\cite{dau2014existence} that, for certain constraints on the zero pattern of a generator matrix (see Definition~\ref{mdscond}), there exist $[n,k]_q$ MDS codes for all prime powers $q \geqslant n+k-1$.
This conjecture (called the GM-MDS conjecture), which stimulated a lot of interest from the community \cite{birs, halbawi2014distributed, heidarzadeh2017algebraic, yan2014weakly, yildiz2018further}, has recently been proved by Lovett~\cite{lovett2018proof} and independently by Yildiz and Hassibi~\cite{yildiz2018optimum}.

For a positive integer $n$, define $[n] := \{1,\dots,n\}$.

\begin{dfn}[MDS Condition]\label{mdscond}
Let $\mathcal{S}=\{S_1,\dots,S_k\}$ be a set system where $S_i\subseteq [n]$ for each $i\in [k]$.
We say that $\mathcal{S}$ satisfies the MDS condition if, for any nonempty $I\subseteq [k]$, we have
$$
|I|+\left| \bigcap_{i\in I} S_i \right| \leqslant k.
$$
\end{dfn}

Note that if $\mathcal{S}=\{S_1,\dots,S_k\}$ satisfies the MDS condition then $|S_i|\leqslant k-1$ for each $i\in[k]$.

Throughout, we use $\mathbb{F}_q$ to denote the finite field with $q$ elements where $q$ is a prime power, $K$ to denote a general field, and $x, x_1, \dots, x_n$ are formal variables.
We use $\mathbb F_q(x_1,\dots,x_n)$ to denote the field of rational functions with variables $x_1,\dots, x_n$ and coefficients from the field $\mathbb{F}_q$ and $\mathbb F_q(x_1,\dots,x_n)[x]$ denotes the ring of (univariate) polynomials over $\mathbb F_q(x_1,\dots,x_n)$.

For positive integers $k$ and $n$, a $k\times n$ $\mathbb{F}_q$-matrix $A$ is called \textit{MDS} if every $k\times k$ submatrix of $A$ is invertible.
Note that a code is MDS if and only if its generator matrices are MDS.
Thus Lovett and Yildiz-Hassibi (independently) proved the following result.

\begin{theorem}[GM-MDS Conjecture in  {\cite{dau2014existence}}]\label{gmmds}
Let $\mathcal{S}=\{S_1,\dots,S_k\}$ be a set system where $S_i\subseteq [n]$ for all $i\in [k]$.
Suppose $\mathcal{S}$ satisfies the MDS condition. 
Then for any finite field $\mathbb{F}_q$ with $q \geqslant n+k-1$, there exists a $k\times n$ MDS matrix $A$ over $\mathbb{F}_q$ with $A_{i,j}=0$ whenever $j\in S_i$.
\end{theorem}

Both the proofs of the GM-MDS conjecture by Lovett and Yildiz-Hassibi use the Schwartz-Zippel Lemma~\cite{schwartz80,zippel79}.

Our main contribution is to provide two constructions of $k \times n$ MDS matrices over $\mathbb F_q$ satisfying a support constraint that is slightly stronger than the MDS condition but with $q \geqslant n$ or $q \geqslant n+1$.
(See Theorem~\ref{thm:main1} and Theorem~\ref{thm:main2}.)
The constructions presented herein
are elementary and, in particular, rely on neither the Schwartz-Zippel lemma nor the GM-MDS conjecture.
Moreover, the values of $x_i$ in our constructions can be chosen arbitrarily as long as they are distinct.

Lovett conjectured a slight generalisation of the GM-MDS conjecture, which we state in Section~\ref{sec:generalized_gm_mds}.
In Section~\ref{sec:counterexamples} and in Section~\ref{sec:special_case}, we show that Lovett's conjecture is false and to what extent it is true.

\section{Main result and application}\label{sec:main_result_and_application}

In this section, we present two constructions of $k \times n$ MDS matrices over $\mathbb F_q$ that have constrained support, where $q = n$ or $q=n+1$.
Our constructions require neither the Schwartz-Zippel lemma nor the GM-MDS conjecture.

Let $P=\{p_1,p_2,\dots,p_k\} \subseteq K[x]$ for some field $K$ (not necessarily finite) where, for each $i\in[k]$, the degree of $p_i$ is at most $k-1$.
We write $p_i(x)=c_{i,1}x^{k-1}+\dots+c_{i,k-1}x+c_{i,k}$ for each $i\in [k]$.
Define the \emph{coefficient matrix} $C(P)$ by $C(P)_{i,j}=c_{i,j}$ for all $i,j\in [k]$.

The proof of the following lemma is standard.

\begin{lem}\label{lem:linindpols}
  The polynomials in $P$ are linearly independent over the field $K$ if and only if the determinant of $C(P)$ is nonzero in $K$.
\end{lem}

Let $S$ be a multiset where all of its elements are from the set $[n]$ and let $\mathbf{0}_{k-1}$ denote the zero (row) vector of size $k-1$.
Define the polynomial $p=p(S)\in \mathbb{F}_q(x_1,\dots,x_n)[x]$ as $p(x)\coloneqq \prod_{i\in S}(x-x_i)$.

\begin{lem}\label{lem:factorisation}
Let $\mathcal{S}=\{S_1,S_2,\dots,S_k\}$ be a multiset system such that, for each $i\in [k]$, we have $|S_i|\leqslant k-1$ and all elements of $S_i$ are from the set $[n]$.
Let $P=\{p(S_1),p(S_2),\dots,p(S_k)\}\subseteq \mathbb{F}_q(x_1,\dots,x_n)[x]$.
Suppose there exists $\xi\in [n]$ such that $\xi\in \bigcap_{i=1}^{k-1}S_i$.
Then
$$
C(P)=\begin{pmatrix}
  C(P') & \mathbf{0}_{k-1}^\top \\
  {\mathbf{0}_{k-1}} & 1
\end{pmatrix} \cdot C(Q)
$$
where $P'=\{p(S_1\backslash\{\xi\}),p(S_2\backslash\{\xi\}),\dots,p(S_{k-1}\backslash\{\xi\})\}$ and $Q=\{x^{k-2}(x-x_{\xi}),x^{k-3}(x-x_{\xi}),\dots,x(x-x_{\xi}),x-x_{\xi},p(S_k)\}$.
\end{lem}

\begin{proof}
For $i\in [k-1]$, let us write
$$
p(S_i\backslash\{\xi\})=\frac{p(S_i)}{x-x_{\xi}}=c_{i,1}x^{k-2}+\dots+c_{i,k-2}x+c_{i,k-1}.
$$
Therefore,
$$
p(S_i)=c_{i,1}\cdot x^{k-2}(x-x_{\xi})+\dots+c_{i,k-1}\cdot(x-x_{\xi})+0\cdot p(S_k).    
$$
The last row of the equation for $C(P)$ follows since 
$$
p(S_k)=0\cdot x^{k-2}(x-x_{\xi})+\dots+0\cdot (x-x_{\xi})+1\cdot p(S_k). \qedhere
$$
\end{proof}

\begin{remark}
The factorisation also generalises to the cases when there are $\lambda$ polynomials that have $k-\lambda$ common roots (counting multiplicity).
\end{remark}

The following lemma yields a useful expression for the determinant of $C(Q)$.

\begin{lem}\label{lem:circulant_det}
Let $\xi\in [n]$ and let $S$ be a multiset where $|S|\leqslant k-1$ and each element of $S$ is from the set $[n]$.
Let $Q=\{x^{k-2}(x-x_{\xi}),x^{k-3}(x-x_{\xi}),\dots,x(x-x_{\xi}),x-x_{\xi},p(S)\}$.
Then
$$
\det\left(C(Q)\right)=\prod_{i\in S} (x_{\xi}-x_i).
$$
In particular, $\det\left(C(Q)\right)$ is nonzero in $\mathbb{F}_q(x_1,\dots,x_n)$ if and only if $\xi\notin S$.
\end{lem}

\begin{proof}
Fix $j\in S$.
Suppose we write $p(S\backslash \{j\})=c_1x^{k-2}+\dots+c_{k-2}x+c_{k-1}$.
Note that $p(S)=(x-x_j)\cdot p(S\backslash \{j\})$.
Hence
\begin{align*}
    &p(S)-\left(c_1x^{k-2}(x-x_{\xi}) + \dots + c_{k-2}x(x-x_{\xi}) + c_{k-1}(x-x_{\xi})\right) \\
    &= p(S)-(x-x_{\xi})(c_1x^{k-2}+\dots+c_{k-2}x+c_{k-1}) \\
    &= (x-x_j)\cdot p(S\backslash\{j\})-(x-x_{\xi})\cdot p(S\backslash\{j\}) \\
    &= (x_{\xi}-x_j)\cdot p(S\backslash\{j\}).
\end{align*}
It follows that
$$
\det\left(C(Q)\right)=(x_{\xi}-x_j)\cdot \det\left(C(Q')\right),
$$
where $Q'=\{x^{k-2}(x-x_{\xi}),x^{k-3}(x-x_{\xi}),\dots,x(x-x_{\xi}),x-x_{\xi},p(S\backslash \{j\})\}$.
We keep repeating this process until we obtain the empty set from $S$ and we obtain
$$
\det\left(C(Q)\right)=\prod_{i\in S} (x_{\xi}-x_i). \qedhere
$$
\end{proof}

Combining Lemma~\ref{lem:factorisation} and Lemma~\ref{lem:circulant_det}, we have Theorem~\ref{thm:triangular_det} below.

\begin{theorem}\label{thm:triangular_det}
Let $\mathcal{S}=\{S_1,S_2,\dots,S_k\}$ be a multiset system such that for each $i\in [k]$, we have $|S_i|\leqslant k-1$ and all elements of $S_i$ are from the set $[n]$.
Let $P=\{p(S_1),p(S_2),\dots,p(S_k)\}\subseteq \mathbb{F}_q(x_1,\dots,x_n)[x]$.
Suppose that, for all $i\in [k]$, we have $\left|\bigcap_{j=1}^{i} S_j\right| \geqslant k-i$.
Suppose also that there exist $\xi_1,\dots,\xi_{k-1}\in [n]$ (not necessarily distinct) such that, for all $i\in [k]$, the multiset $T_{i}=\{\xi_1,\dots,\xi_{k-i}\}$ is contained in the intersection $\bigcap_{j=1}^i S_j$.
Then
$$
\det\left(C(P)\right)=\prod_{i=2}^k \prod_{j\in S_i\backslash T_i} (x_{\xi_{k-i+1}}-x_j).
$$
\end{theorem}

\begin{proof}
Note that $\xi_i\in \bigcap_{j=1}^{k-i} S_j$ for all $i\in [k-1]$.
Using Lemma~\ref{lem:factorisation}, we repeatedly factorise $\det\left(C(P)\right)$ and Lemma~\ref{lem:circulant_det} gives us the formula for each factor.
\end{proof}


We will utilise the construction of Reed-Solomon codes for our next results (see~\cite{halbawi2016balanced}).
A Reed-Solomon code of length $n$ and dimension $k$ over finite field $\mathbb{F}_q$ is the $k$-dimensional subspace of $\mathbb{F}_q^n$ given by ${\mathcal{C}}_{\text{RS}}=\{(p(a_1),\dots,p(a_n)):\text{deg}(p(x))<k\}$, where $p(x)$ are polynomials over $\mathbb{F}_q$ with degree less than $k$ and the evaluation points $a_1,\dots,a_n \in \mathbb{F}_q$ are all distinct.
The codeword associated with $p(x)$ is $(p(a_1),\dots,p(a_n))$.
Since a Reed-Solomon code is an MDS code, any of its generator matrices is an MDS matrix.

Let $p_1,\dots,p_k$ be polynomials over $\mathbb{F}_q$ with degree less than $k$.
Given an $n$-subset $\{a_1,\dots,a_n\}$ of $\mathbb{F}_q$, define the matrix $A=A(\{a_1,\dots,a_n\})$ by $A_{i,j}=p_i(a_j)$ for all $i \in [k]$ and $j \in [n]$.
If $p_1,\dots,p_k$ are linearly independent over $\mathbb F_q$ then, for any $n$-subset $\{a_1,\dots,a_n\}$ of $\mathbb{F}_q$, the matrix $A = A(\{a_1,\dots,a_n\})$ is a generator matrix of a Reed-Solomon code.
In particular $A$ is an MDS matrix. 
Using Lemma~\ref{lem:linindpols}, we have that if the determinant of $C(\{p_1,\dots,p_k\})$ is nonzero in $\mathbb{F}_q$ then, for any $n$-subset $\{a_1,\dots,a_n\}$ of $\mathbb{F}_q$, the matrix $A(\{a_1,\dots,a_n\})$ is an MDS matrix.

Now we can state our first main result.

\begin{theorem}\label{thm:main1}
Let $\mathcal{S}=\{S_1,\dots,S_k\}$ be a set system where $S_i\subseteq [n]$ for all $i\in [k]$.
Suppose $\mathcal{S}$ satisfies the MDS condition and, for all $i\in [k]$, we have $\left|\bigcap_{j=1}^{i} S_j\right|=k-i$.
Then for any finite field $\mathbb{F}_q$ with $q \geqslant n$, there exists a $k\times n$ MDS matrix $A$ over $\mathbb{F}_q$ such that $A_{i,j}=0$ if and only if $j\in S_i$.
\end{theorem}

\begin{proof}
Let $\mathbb{F}_q$ be a finite field with $q\geqslant n$.
Without loss of generality, assume that $\bigcap_{j=1}^{i} S_j = [k-i]$ for all $i\in [k]$.
Let $P=\{p(S_1),p(S_2),\dots,p(S_k)\}$.
By Theorem~\ref{thm:triangular_det}, we have
$$
\det\left(C(P)\right)=\prod_{i=2}^k \prod_{j\in S_i\backslash [k-i]} (x_{k-i+1}-x_j).
$$
For $2\leqslant i\leqslant k$, we have $k-i+1\notin S_i\backslash [k-i]$ since $\mathcal{S}$ satisfies the MDS condition.
Note that the value of $\det\left(C(P)\right)$ will be nonzero in $\mathbb{F}_q$ as long as we substitute distinct elements of $\mathbb{F}_q$ for $x_1,x_2,\dots,x_n$.

Fix a subset $\{a_1,a_2,\dots,a_n\} \subseteq \mathbb{F}_q$.
For all $i\in [n]$, we set $x_i=a_i$.
Under this substitution 
we obtain a new set of polynomials $R=\{r_1,r_2,\dots,r_k\}\subseteq \mathbb{F}_q[x]$ where $r_i(x)=\prod_{\lambda\in S_i}(x-a_{\lambda})$ for each $i\in [k]$.
Hence $\det\left(C(R)\right)$ is nonzero in $\mathbb{F}_q$ and therefore $R$ is linearly independent over $\mathbb{F}_q$.

Now we can construct a $k\times n$ MDS matrix $A$ where $A_{i,j}=r_i(a_j)$ for all $i\in[k]$ and $j\in [n]$.
Let $i\in[k]$ and let $j\in [n]$.
If $j\in S_i$ then obviously $A_{i,j}=r_i(a_j)=0$.
Suppose $j\notin S_i$.
Then $A_{i,j}=r_i(a_j)=\prod_{\lambda\in S_i}(a_j-a_{\lambda})$ is nonzero.
Therefore, we also have that $A_{i,j}=0$ if and only if $j\in S_i$.
\end{proof}

\begin{exa}
Let $n=7$ and let $\mathcal{S}=\{S_1,S_2,S_3,S_4\}$ where $S_1=\{1,2,3\}$, $S_2=\{1,2,6\}$, $S_3=\{1,5,7\}$, and $S_4=\{3,4,5\}$.
Note that $|S_1|=3$, $|S_1\cap S_2|=2$, $|S_1\cap S_2\cap S_3|=1$, and $|S_1\cap S_2\cap S_3\cap S_4|=0$.
For $q=n=7$, we will construct a $4\times 7$ MDS matrix $A$ over $\mathbb{F}_7$ such that $A_{i,j}=0$ if and only if $j\in S_i$.
Let $P=\{p_1,p_2,p_3,p_4\}$ where
\begin{align*}
p_1(x) &= (x-x_1)(x-x_2)(x-x_3) \\
p_2(x) &= (x-x_1)(x-x_2)(x-x_6) \\
p_3(x) &= (x-x_1)(x-x_5)(x-x_7) \\
p_4(x) &= (x-x_3)(x-x_4)(x-x_5).
\end{align*}
We have
\begin{align*}
\det\left(C(P)\right)&=\det\begin{pmatrix}
1 & -(x_1+x_2+x_3) & x_1x_2+x_1x_3+x_2x_3 & -x_1x_2x_3 \\
1 & -(x_1+x_2+x_6) & x_1x_2+x_1x_6+x_2x_6 & -x_1x_2x_6 \\
1 & -(x_1+x_5+x_7) & x_1x_5+x_1x_7+x_5x_7 & -x_1x_5x_7 \\
1 & -(x_3+x_4+x_5) & x_3x_4+x_3x_5+x_4x_5 & -x_3x_4x_5
\end{pmatrix}\\
&= (x_1-x_3)(x_1-x_4)(x_1-x_5)(x_2-x_5)(x_2-x_7)(x_3-x_6).
\end{align*}

Note that the value of $\det\left(C(P)\right)$ will be nonzero in $\mathbb{F}_7$ as long as we substitute distinct elements of $\mathbb{F}_7$ for $x_1,x_2,\dots,x_7$.
Let $\mathbb{F}_q=\mathbb{Z}_7 = \{a_1,\dots,a_7\}$
where $a_1=1$, $a_2=0$, $a_3=3$, $a_4=6$, $a_5=5$, $a_6=2$, $a_7=4$.
For all $i\in [7]$, we set $x_i=a_i$.
Under this substitution, we have a new set of polynomials $R=\{r_1,r_2,r_3,r_4\}$ where $r_1(x)=(x-1)x(x-3)$, $r_2(x)=(x-1)x(x-2)$, $r_3(x)=(x-1)(x-5)(x-4)$, and $r_4(x)=(x-3)(x-6)(x-5)$.
The determinant of $C(R)$ is nonzero in $\mathbb{Z}_7$ and our $4\times 7$ MDS matrix $A$ over $\mathbb{Z}_7$ is
$$
A=
\begin{pmatrix}
0 & 0 & 0 & 6 & 5 & 5 & 5 \\
0 & 0 & 6 & 1 & 4 & 0 & 3 \\
0 & 1 & 4 & 3 & 0 & 6 & 0 \\
2 & 1 & 0 & 0 & 0 & 2 & 2 \\
\end{pmatrix}.
$$
We remark that the top left hand of $A$ is a triangle of zeros; this is the structure imposed by the assumption of Theorem~\ref{thm:main1}.
\end{exa}

Using ideas similar to those used in the proof of Theorem~\ref{thm:main1}, we establish our second main result.

\begin{theorem}\label{thm:main2}
Let $\mathcal{S}=\{S_1,\dots,S_k\}$ be a set system where $S_i\subseteq [n]$ for all $i\in [k]$.
Suppose for all $i\in [k]$, we have $|S_i|\leqslant i-1$.
Then for any finite field $\mathbb{F}_q$ with $q\geqslant n+1$, there exists a $k\times n$ MDS matrix $A$ over $\mathbb{F}_q$ such that $A_{i,j}=0$ if and only if $j\in S_i$.
\end{theorem}

\begin{proof}
Let $\mathbb{F}_q$ be a finite field with $q\geqslant n+1$.
Observe that $|S_i|\leqslant i-1$ for all $i\in [k]$ implies that $\mathcal{S}$ satisfies the MDS condition.
For each $i\in [k]$, define the multiset $U_i=S_i\cup \{\underbrace{n+1,\dots,n+1}_{k-i}\}$ so $|U_i|\leqslant k-1$.
Let $P=\{p(U_1),p(U_2),\dots,p(U_k)\}$.
Note that $\bigcap_{j=1}^i U_j=\{\underbrace{n+1,\dots,n+1}_{k-i}\}$ for all $i\in [k]$.
By Theorem~\ref{thm:triangular_det}, we have
$$
\det\left(C(P)\right)=\prod_{i=2}^k\prod_{j\in S_i}(x_{n+1}-x_j).
$$
For $2\leqslant i\leqslant k$ and $j\in S_i$, we clearly have $j\neq n+1$.
Note that the value of $\det\left(C(P)\right)$ will be nonzero in $\mathbb{F}_q$ as long as we substitute distinct elements of $\mathbb{F}_q$ for $x_1,x_2,\dots,x_{n+1}$.

Fix a subset $\{a_1,a_2,\dots,a_{n+1}\} \subseteq \mathbb{F}_q$.
For all $i\in [n+1]$, we set $x_i=a_i$.
Under this substitution 
we obtain a new set of polynomials $R=\{r_1,r_2,\dots,r_k\}\subseteq \mathbb{F}_q[x]$ where $r_i(x)=(x-a_{n+1})^{k-i}\prod_{\lambda\in S_i}(x-a_{\lambda})$ for each $i\in [k]$.
Hence $\det\left(C(R)\right)$ is nonzero in $\mathbb{F}_q$ and therefore $R$ is linearly independent over $\mathbb{F}_q$.

Now we can construct a $k\times n$ MDS matrix $A$ where $A_{i,j}=r_i(a_j)$ for all $i\in[k]$ and $j\in [n]$.
Let $i\in[k]$ and let $j\in [n]$.
If $j\in S_i$ then obviously $A_{i,j}=r_i(a_j)=0$.
Suppose $j\notin S_i$.
Thus $A_{i,j}=r_i(a_j)=(a_j-a_{n+1})^{k-i}\prod_{\lambda\in S_i}(a_j-a_{\lambda})$ is nonzero.
Therefore, we also have that $A_{i,j}=0$ if and only if $j\in S_i$.
\end{proof}

\begin{exa}
Let $n=6$ and let $\mathcal{S}=\{S_1,S_2,S_3,S_4\}$ where $S_1=\emptyset$, $S_2=\{3\}$, $S_3=\{2,5\}$, and $S_4=\{1,4,6\}$.
Note that $|S_1|=0$, $|S_2|\leqslant 1$, $|S_3|\leqslant 2$, and $|S_4|\leqslant 3$.
For $q=n+1=7$, we will construct a $4\times 6$ MDS matrix $A$ over $\mathbb{F}_7$ such that $A_{i,j}=0$ if and only if $j\in S_i$.

Let $U_1=\{7,7,7\}$, $U_2=\{3,7,7\}$, $U_3=\{2,5,7\}$, and $U_4=\{1,4,6\}$.
Let $P=\{p_1,p_2,p_3,p_4\}$ where
\begin{align*}
p_1(x) &= (x-x_7)^3 \\
p_2(x) &= (x-x_3)(x-x_7)^2 \\
p_3(x) &= (x-x_2)(x-x_5)(x-x_7)\\
p_4(x) &= (x-x_1)(x-x_4)(x-x_6).
\end{align*}
We have
\begin{align*}
\det\left(C(P)\right)&=\det\begin{pmatrix}
1 & -3x_7 & 3{x_7}^2 & -{x_7}^3 \\
1 & -(x_3+2x_7) & {x_7}^2+2x_3x_7 & -x_3{x_7}^2 \\
1 & -(x_2+x_5+x_7) & x_2x_5+x_2x_7+x_5x_7 & -x_2x_5x_7 \\
1 & -(x_1+x_4+x_6) & x_1x_4+x_1x_6+x_4x_6 & -x_1x_4x_6
\end{pmatrix}\\
&= (x_7-x_1)(x_7-x_2)(x_7-x_3)(x_7-x_4)(x_7-x_5)(x_7-x_6).
\end{align*}

Note that the value of $\det\left(C(P)\right)$ will be nonzero in $\mathbb{F}_7$ as long as we substitute distinct elements of $\mathbb{F}_7$ for $x_1,x_2,\dots,x_7$.
Let $\mathbb{F}_q=\mathbb{Z}_7=\{a_1,a_2,\dots,a_7\}$ where $a_1=6$, $a_2=4$, $a_3=0$, $a_4=3$, $a_5=2$, $a_6=5$, $a_7=1$.
For all $i\in [7]$, we set $x_i=a_i$.
Under this substitution, we have a new set of polynomials $R=\{r_1,r_2,r_3,r_4\}$ where $r_1(x)=(x-1)^3$, $r_2(x)=x(x-1)^2$, $r_3(x)=(x-4)(x-2)(x-1)$, $r_4(x)=(x-6)(x-3)(x-5)$.
The determinant of $C(R)$ is nonzero in $\mathbb{Z}_7$ and our $4\times 7$ MDS matrix $A$ over $\mathbb{Z}_7$ is
$$
A=
\begin{pmatrix}
 6 & 6 & 6 & 1 & 1 & 1 \\
 3 & 1 & 0 & 5 & 2 & 3 \\
 5 & 0 & 6 & 5 & 0 & 5 \\
 0 & 2 & 1 & 0 & 2 & 0 \\
\end{pmatrix}.
$$
We remark that the $i$-th row of $A$ has at most $i-1$ zeros for $i \in \{1,\dots,4\}$; this is the structure imposed by the assumption of Theorem~\ref{thm:main2}.
\end{exa}

To complete this section, we consider an example that motivates a search for possible extensions of Theorem~\ref{thm:main1} and Theorem~\ref{thm:main2}.

\begin{exa}
Let $n=7$ and let $\mathcal{S}=\{S_1,S_2,S_3,S_4\}$ where $S_1=\{1,5,6\}$, $S_2=\{1,3,5\}$, $S_3=\{2,6,7\}$, and $S_4=\{2,4,7\}$.
Let $P=\{p_1,p_2,p_3,p_4\}$ where
\begin{align*}
p_1(x) &= (x-x_1)(x-x_5)(x-x_6) \\
p_2(x) &= (x-x_1)(x-x_3)(x-x_5) \\
p_3(x) &= (x-x_2)(x-x_6)(x-x_7) \\
p_4(x) &= (x-x_2)(x-x_4)(x-x_7).
\end{align*}
Here $\mathcal{S}$ satisfies the MDS condition but it does not satisfy the assumptions of Theorem~\ref{thm:main1} or Theorem~\ref{thm:main2}.
However, the determinant of $C(P)$ splits into linear factors just like in the previous two examples.
Indeed, we have
\begin{align*}
\det\left(C(P)\right)&=\det\begin{pmatrix}
1 & -(x_1+x_5+x_6) & x_1x_5+x_1x_6+x_5x_6 & -x_1x_5x_6 \\
1 & -(x_1+x_3+x_5) & x_1x_3+x_1x_5+x_3x_5 & -x_1x_3x_5 \\
1 & -(x_2+x_6+x_7) & x_2x_6+x_2x_7+x_6x_7 & -x_2x_6x_7 \\
1 & -(x_2+x_4+x_7) & x_2x_4+x_2x_7+x_4x_7 & -x_2x_4x_7
\end{pmatrix}\\
&= (x_1-x_2)(x_1-x_7)(x_2-x_5)(x_3-x_6)(x_5-x_7)(x_6-x_4).
\end{align*}

Note that the value of $\det\left(C(P)\right)$ will be nonzero in $\mathbb{F}_q$ as long as we substitute distinct elements of $\mathbb{F}_q$ for $x_1,x_2,\dots,x_7$.
Consequently, for any finite field $\mathbb{F}_q$ with $q\geqslant 7$, we can construct a $4\times 7$ MDS matrix $A$ over $\mathbb{F}_q$ such that $A_{i,j}=0$ if and only if $j\in S_i$.
As an example, we will construct such MDS matrix $A$ for $q=7$.
Let $\mathbb{F}_q=\mathbb{Z}_7 = \{a_1,\dots,a_7\}$
where $a_1=2$, $a_2=5$, $a_3=0$, $a_4=1$, $a_5=4$, $a_6=3$, $a_7=6$.
For all $i\in [7]$, we set $x_i=a_i$.
Under this substitution, we have a new set of polynomials $R=\{r_1,r_2,r_3,r_4\}$ where $r_1(x)=(x-2)(x-4)(x-3)$, $r_2(x)=(x-2)x(x-4)$, $r_3(x)=(x-5)(x-3)(x-6)$, $r_4(x)=(x-5)(x-1)(x-6)$.
The determinant of $C(R)$ is nonzero in $\mathbb{Z}_7$ and our $4\times 7$ MDS matrix $A$ over $\mathbb{Z}_7$ is
$$
A=
\begin{pmatrix}
0 & 6 & 4 & 1 & 0 & 0 & 3 \\
0 & 1 & 0 & 3 & 0 & 4 & 6 \\
2 & 0 & 1 & 2 & 2 & 0 & 0 \\
5 & 0 & 5 & 0 & 6 & 5 & 0 \\
\end{pmatrix}.
$$

Furthermore, the matrix $A$ is sparsest and balanced in the sense discussed in \cite{dau2013balanced,halbawi2016balanced,song2018generalized}.
This example suggests that it would be interesting to study sets of polynomials having the property that the determinant of the coefficient matrix splits into linear factors.
\end{exa}


\section{A generalisation of the GM-MDS Conjecture}\label{sec:generalized_gm_mds}

In this section we introduce the conjecture of Lovett~\cite[Conjecture 1.5]{lovett2018proof}.
Let $\mathbb{F}_q$ be a finite field where $q$ is a prime power and let $\mathbb{N}=\{0,1,2,\dots\}$ be the set of non-negative integers.
Let $\mathbf{v}\in \mathbb{N}^n$ be a vector. 
Denote by $\mathbf{v}(i)$ the $i$-th coordinate of $\mathbf{v}$ and define $|\mathbf{v}| \coloneqq \sum_{i=1}^n \mathbf{v}(i)$. 
For a set of vectors $\mathcal{V}=\{\mathbf{v}_1,\dots,\mathbf{v}_m\}$ and a subset $I\subseteq [m]$, define $\mu_{\mathcal{V}}(I)$ as
$$
\mu_{\mathcal{V}}(I)=(\text{min}_{i\in I}\mathbf{v}_i(1),\dots,\text{min}_{i\in I}\mathbf{v}_i(n)).
$$
Given a parameter $k>|\mathbf{v}|$, define a set of polynomials in $\mathbb{F}_q(x_1,\dots,x_n)[x]$:
$$
P(k,\mathbf{v}) :=\left\{ \prod_{j\in [n]}(x-x_j)^{\mathbf{v}(j)}x^e : e=0,\dots,k-1-|\mathbf{v}| \right\}.
$$
For a set of vectors $\mathcal{V}=\{\mathbf{v}_1,\dots,\mathbf{v}_m\}\subseteq \mathbb{N}^n$ we define the (multi)set
$$
P(k,\mathcal{V})=P(k,\mathbf{v}_1)\cup\dots\cup P(k,\mathbf{v}_m).
$$
Observe that
$$
|P(k,\mathcal{V})|=|P(k,\mathbf{v}_1)|+\dots+|P(k,\mathbf{v}_m)|.
$$
\begin{dfn}[Property $V(k)$ {\cite[Definition 1.4]{lovett2018proof}}]\label{dfn}
Let $k,m,n\geqslant 1$ be integers and let $\mathcal{V}=\{ \mathbf{v}_1,\dots,\mathbf{v}_m \} \subseteq \mathbb{N}^n$. 
We say that $\mathcal{V}$ satisfies $V(k)$ if it satisfies:
\begin{enumerate}[label=(\Roman*)]
\item $|\mathbf{v}_i|\leqslant k-1$ for all $i\in [m]$.
\item For all $I\subseteq [m]$ nonempty, $\sum_{i\in I}(k-|\mathbf{v}_i|)+|\mu_{\mathcal{V}}(I)|\leqslant k$.
\end{enumerate}
\end{dfn}

The definition of property $V_l(k)$ below is a slight modification of Definition 1.6 (Property $V^*(k)$) in \cite{lovett2018proof}.

\begin{dfn}[Property $V_l(k)$]\label{dfnl}
Let $k,m,n\geqslant 1$ and $l\geqslant 0$ be integers where $n\geqslant l$ and let $\mathcal{V}=\{ \mathbf{v}_1,\dots,\mathbf{v}_m \} \subseteq \mathbb{N}^n$. 
We say that $\mathcal{V}$ satisfies $V_l(k)$ if it satisfies $V(k)$, and additionally it satisfies:
\begin{enumerate}[label=(\Roman*)]
\setcounter{enumi}{2}
\item $\mathbf{v}_i \in \left \{ 0,1 \right \}^{n-l} \times \mathbb{N}^l$ for all $i\in [m]$.
\end{enumerate}
\end{dfn}
Note that if $\mathcal{V}$ satisfies $V_l(k)$, then $\mathcal{V}$ satisfies $V_{l'}(k)$ for any $l'$ where $l\leqslant l'\leqslant n$. The remainder of this paper is on the following conjecture of Shachar Lovett.

\begin{conj}[Conjecture 1.5 in {\cite{lovett2018proof}}]\label{conj}
Let $k,m,n\geqslant 1$ be integers and let $\mathcal{V}=\{ \mathbf{v}_1,\dots,\mathbf{v}_m \} \subseteq \mathbb{N}^n$. 
Assume that $\mathcal{V}$ satisfies $V(k)$. 
Then the polynomials in $P(k,\mathcal{V})$ are linearly independent over $\mathbb{F}_q(x_1,\dots,x_n)$.
\end{conj}
We will show that Conjecture~\ref{conj} is false in general.

Note that if $\mathcal{V}$ satisfies $V(k)$ then $\mathcal{V}$ satisfies $V_l(k)$ for some $l\leqslant n$. Conjecture~\ref{conjl} below is an analogous formulation of Conjecture~\ref{conj} in terms of property $V_l(k)$.

\begin{conj}\label{conjl}
Let $k,m,n\geqslant 1$ and $l\geqslant 0$ be integers where $n\geqslant l$ and let $\mathcal{V}=\{ \mathbf{v}_1,\dots,\mathbf{v}_m \} \subseteq \mathbb{N}^n$. 
Assume that $\mathcal{V}$ satisfies $V_l(k)$. 
Then the polynomials in $P(k,\mathcal{V})$ are linearly independent over $\mathbb{F}_q(x_1,\dots,x_n)$.
\end{conj}

The case $l=0$ corresponds to the GM-MDS conjecture since multiple roots are not included in this case.
Lovett proved Theorem~\ref{thml1} below, which corresponds to $l=1$.

\begin{theorem}[See Theorem 1.7 in {\cite{lovett2018proof}}]\label{thml1}
Let $k,m,n\geqslant 1$ be integers and let $\mathcal{V}=\{ \mathbf{v}_1,\dots,\mathbf{v}_m \} \subseteq \mathbb{N}^n$. 
Assume that $\mathcal{V}$ satisfies $V_1(k)$. 
Then the polynomials in $P(k,\mathcal{V})$ are linearly independent over $\mathbb{F}_q(x_1,\dots,x_n)$.
\end{theorem}

We will show that Conjecture~\ref{conjl} is true for $l=2$ (see Section~\ref{sec:special_case}).
However, in Section~\ref{sec:counterexamples} below, we provide constructions for counterexamples to  Conjecture~\ref{conjl} for all $l\geqslant 3$.

\section{Counterexamples to Conjecture~\ref{conj}}\label{sec:counterexamples}
In this section we show that Conjecture~\ref{conj} is false in general. 
More precisely, we show that, for all $l\geqslant 3$, there exists $\mathcal{V} \subseteq \mathbb{N}^n$ for some $n\geqslant l$ such that $\mathcal{V}$ satisfies $V_l(k)$ for some $k\geqslant 1$, but $P(k,\mathcal{V})$ is linearly dependent over $\mathbb{F}_q(x_1,\dots,x_n)$. 

Let $k=m=2b$ and $n=2b-1$ where $b\geqslant 2$ is an integer.
Let $\mathcal{W}_b=\{ \mathbf{w}_1,\dots,\mathbf{w}_{2b} \} \subseteq \mathbb{N}^{2b-1}$ where
$$
\mathbf{w}_i(j)=
\begin{cases}
2b-1 & \text{if } j=i, \\
0  & \text{otherwise},
\end{cases}
$$
for $i\in [2b-1]$ and $\mathbf{w}_{2b}=(1,1,\dots,1)$.
We will show that the set $\mathcal{W}_b$ satisfies $V_{2b-1}(2b)$ but the polynomials in $P(2b,\mathcal{W}_b)$ are linearly dependent over $\mathbb{F}_q(x_1,x_2,\dots,x_{2b-1})$.
Note that $q$ does not depend on $b$ so $\mathbb{F}_q$ could be any finite field.

First we show that the set $\mathcal{W}_b$ satisfies $V_{2b-1}(2b)$.

\begin{prop}\label{prop1}
Let $b\geqslant 2$ be an integer.
Then the set $\mathcal{W}_b$ satisfies $V_{2b-1}(2b)$.
\end{prop}

\begin{proof}
For any $i\in [2b]$, we have $|\mathbf{w}_i|=2b-1$ so $\mathcal{W}_b$ satisfies (I). 
Since $\mathcal{W}_b \subseteq \mathbb{N}^{2b-1}$ then $\mathcal{W}_b$ also satisfies (III). 
Now let $I\subseteq [2b]$. 
Note that (II) always holds if $|I|=1$ so assume $|I|>1$. 
Suppose $|I|=2$. 
If $2b\notin I$ then $\mu_{\mathcal{W}_b}(I)=\mathbf{0}$ and hence
$$
\sum_{i\in I}(2b-|\mathbf{w}_i|)+|\mu_{\mathcal{W}_b}(I)|=|I|+0=2\leqslant 2b.
$$
If $2b\in I$ then $|\mu_{\mathcal{W}_b}(I)|=1$ and hence
$$
\sum_{i\in I}(2b-|\mathbf{w}_i|)+|\mu_{\mathcal{W}_b}(I)|=|I|+1=3\leqslant 2b.
$$
Suppose $3\leqslant |I|\leqslant 2b$. Then $\mu_{\mathcal{W}_b}(I)=\mathbf{0}$ and hence
$$
\sum_{i\in I}(2b-|\mathbf{w}_i|)+|\mu_{\mathcal{W}_b}(I)|=|I|+0=|I|\leqslant 2b.
$$
In any case, $\mathcal{W}_b$ satisfies (II).
Therefore, $\mathcal{W}_b$ satisfies $V_{2b-1}(2b)$.
\end{proof}

Let $x_1,x_2,\dots,x_{2b-1}$ be formal variables and consider the polynomials $p_1, \dots, p_{2b} \in \mathbb{F}_q(x_1,\dots,x_{2b-1})[x]$, where $p_i(x) = (x-x_i)^{2b-1}$ for $i \in [2b-1]$ and $p_{2b}(x) =\prod_{i=1}^{2b-1}(x-x_i)$.
Clearly, we have $P(2b,\mathcal{W}_b)=\left \{ p_1,\dots,p_{2b-1},p_{2b} \right \}$.

\begin{prop}\label{prop2}
The polynomials in $P(2b,\mathcal{W}_b)$ are linearly dependent over $\mathbb{F}_q(x_1,x_2,\dots,x_{2b-1})$.
\end{prop}

In fact, in the above proposition, the finite field $\mathbb{F}_q$ can be replaced by the ring of integers $\mathbb{Z}$.

\begin{proof}
Let $e_{j-1}$ denote the elementary symmetric polynomial in the $2b-1$ variables $x_1,\dots,x_{2b-1}$ with degree $j-1$ where $j\in [2b]$. Then the polynomials in $P(2b,\mathcal{W}_b)$ are linearly dependent over $\mathbb{F}_q(x_1,\dots,x_{2b-1})$ if and only if the rank of the following $2b\times 2b$ matrix
$$
M=C\left(P(2b,\mathcal{W}_b)\right)=
\begin{pmatrix}
  1 & -(2b-1)x_1 & \cdots & -{x_1}^{2b-1} \\
  1 & -(2b-1)x_2 & \cdots & -{x_2}^{2b-1} \\
  \vdots  & \vdots  & \ddots & \vdots  \\
  1 & -(2b-1)x_{2b-1} & \cdots & -{x_{2b-1}}^{2b-1} \\
  e_0 & -e_1 & \cdots & -e_{2b-1}
 \end{pmatrix}
$$
is less than $2b$.
Alternatively, for all $i,j\in [2b]$, we can write the elements of $M$ as
$$
M_{ij}=
\begin{cases}
\binom{2b-1}{j-1} (-x_i)^{j-1} & \text{if } 1\leqslant i\leqslant 2b-1, \\
   (-1)^{j-1}  e_{j-1}       & \text{if } i=2b.
\end{cases}
$$
For $j\in [b]$, define $c(j)$ to be the integer such that 
$$
\binom{2b-1}{j-1}c(j)=\text{lcm}\left(\binom{2b-1}{0},\binom{2b-1}{1},\dots,\binom{2b-1}{b-1}\right).
$$
And, for $j\in \left\{ b+1,\dots,2b \right\}$, let $c(j)=c(2b+1-j)$.
It follows that $\text{gcd}\left(c(1),c(2),\dots,c(2b)\right)=1$.
Now we define a $2b\times 1$ vector $\mathbf{u}$ where
$$
\mathbf{u}(j)= c(j) e_{2b-j}   
$$
for all $j\in [2b]$.
Note that, over any finite field $\mathbb{F}_q$, not all of the $c(j)$ can be zero since $\text{gcd}\left(c(1),c(2),\dots,c(2b)\right)=1$.
Next we will show that $\mathbf u$ is a (right) null vector of $M$.

Let $i=2b$. Then 
$$
\sum_{j=1}^{2b}M_{ij} \mathbf{u}(j)=\sum_{j=1}^{2b}(-1)^{j-1}  e_{j-1} \mathbf{u}(j).
$$
Let $j\in [b]$ so $2b+1-j\in \left\{b+1,\dots,2b \right\}$.
It is clear that $c(j)=c(2b+1-j)$ and, since $j-1$ and $2b-j$ have opposite parity, we have
\begin{align*}
&(-1)^{j-1} e_{j-1} \mathbf{u}(j)+(-1)^{2b-j}  e_{2b-j} \mathbf{u}(2b+1-j)=0.
\end{align*}
Therefore, if $i=2b$ then
\begin{align}
\sum_{j=1}^{2b}M_{ij} \mathbf{u}(j)&=\sum_{j=1}^{2b}(-1)^{j-1}  e_{j-1} \mathbf{u}(j) \nonumber\\
&= \sum_{j=1}^{b}\left[ (-1)^{j-1}  e_{j-1} \mathbf{u}(j)+(-1)^{2b-j}  e_{2b-j} \mathbf{u}(2b+1-j)\right] \nonumber\\
&= 0.
\end{align}
On the other hand, let $i\in [2b-1]$. Then
\begin{align*}
\sum_{j=1}^{2b}M_{ij} \mathbf{u}(j)&=\sum_{j=1}^{2b}\binom{2b-1}{j-1} (-x_i)^{j-1}\mathbf{u}(j) \\
&= \sum_{j=1}^{2b}\binom{2b-1}{j-1} c(j) (-x_i)^{j-1} e_{2b-j}.
\end{align*}
Using the symmetry of binomial coefficients, we have
$$
\binom{2b-1}{j-1}c(j)=\text{lcm}\left(\binom{2b-1}{0},\binom{2b-1}{1},\dots,\binom{2b-1}{b-1}\right)
$$
for all $j\in [2b]$.
Thus, we obtain
$$
\sum_{j=1}^{2b}M_{ij} \mathbf{u}(j)=\text{lcm}\left(\binom{2b-1}{0},\binom{2b-1}{1},\dots,\binom{2b-1}{b-1}\right) \sum_{j=1}^{2b}(-x_i)^{j-1} e_{2b-j}.
$$
Observe that
$$
\sum_{j=1}^{2b}(-x_i)^{j-1} e_{2b-j}=(x_1-x_i)(x_2-x_i)\cdots (x_{2b-1}-x_i)=0
$$
since $i\in [2b-1]$.
Therefore, if $i\in [2b-1]$ then
\begin{align}
\sum_{j=1}^{2b}M_{ij} \mathbf{u}(j)=0.
\end{align}
Combining $(1)$ and $(2)$, we obtain $M\mathbf{u}=\mathbf{0}$ and hence $\mathbf{u}$ is a null vector of the matrix $M$.
Therefore, we conclude that the rank of $M$ is less than $2b$, which means that the polynomials in $P(2b,\mathcal{W}_b)$ are linearly dependent over $\mathbb{F}_q(x_1,\dots,x_{2b-1})$.
\end{proof}

\begin{exa}\label{exab2}
By Proposition~\ref{prop1}, the set $\mathcal{W}_2=\{(3,0,0)$, $(0,3,0)$, $(0,0,3)$, $(1,1,1) \}$ satisfies $V_3(4)$ but, by Proposition~\ref{prop2}, the polynomials $(x-x_1)^3$,$(x-x_2)^3$,$(x-x_3)^3$,$(x-x_1)(x-x_2)(x-x_3)$ are linearly dependent over $\mathbb{F}_q(x_1,x_2,x_3)$.
This gives us a counterexample to Conjecture~\ref{conj} and Conjecture~\ref{conjl} for $l = 3$.
\end{exa}

\begin{exa}\label{otherexa}
Let $\mathcal{Y}=\{(1,3,0,0),(1,0,3,0),(1,0,0,3),(1,1,1,1)\}$.
It is easy to check that $\mathcal{Y}$ satisfies $V_l(5)$ for $l=3,4$.
However, from Example~\ref{exab2}, it follows that the polynomials $(x-x_1)(x-x_2)^3$,$(x-x_1)(x-x_3)^3$,$(x-x_1)(x-x_4)^3$,$(x-x_1)(x-x_2)(x-x_3)(x-x_4)$ are linearly dependent over $\mathbb{F}_q(x_1,x_2,x_3,x_4)$.
\end{exa}

In general, given any $l\geqslant 3$, let $m=4$ and take any $k,n$ where $n\geqslant l$ and $n=k-1$.
Let $$\mathcal{V}=\{(\underbrace{1,\dots,1}_{n-3},3,0,0),\allowbreak (\underbrace{1,\dots,1}_{n-3},0,3,0),\allowbreak (\underbrace{1,\dots,1}_{n-3},0,0,3),\allowbreak (\underbrace{1,\dots,1}_n)\}.$$
We have that $\mathcal{V}$ satisfies $V_l(k)$ but the polynomials in $P(k,\mathcal{V})$ are linearly dependent over $\mathbb{F}_q(x_1,x_2,\dots,x_n)$.
We could also apply similar construction for other values of $b>2$.

\section{The special case of Conjecture~\ref{conjl} when $l=2$}\label{sec:special_case}
In this section we show that the special case of Conjecture~\ref{conjl} is true for $l\leqslant 2$.
We will prove the following.

\begin{theorem}\label{thml2}
Let $k,m\geqslant 1$ and $n\geqslant 2$ be integers and let $\mathcal{V}=\{ \mathbf{v}_1,\dots,\allowbreak \mathbf{v}_m \} \subseteq \mathbb{N}^n$.
Assume that $\mathcal{V}$ satisfies $V_2(k)$.
Then the polynomials in $P(k,\mathcal{V})$ are linearly independent over $\mathbb{F}_q(x_1,\dots,x_n)$.
\end{theorem}

To prove Theorem~\ref{thml2}, following \cite{lovett2018proof}, we will apply the method of minimal counterexample. 
The minimality here is with respect to the parameters $(n,k,m,d)$ in the lexicographical order where $d=|P(k,\mathcal{V})|$. 
We will use the lemmas below to complete our proof of Theorem~\ref{thml2}.
We omit the proofs of Lemma~\ref{lem21}, Lemma~\ref{lem22}, and Lemma~\ref{lem24}, which are very similar to the proofs of the corresponding lemmas in \cite{lovett2018proof}.

Given two vectors $\mathbf{v},\mathbf{w}\in \mathbb{N}^n$ we write $\mathbf{v}\leqslant \mathbf{w}$ if $\mathbf{v}(i)\leqslant \mathbf{w}(i)$ for all $i\in [n]$.

We require three lemmas proved in \cite{lovett2018proof}.

\begin{lem}[See Lemma 2.1 in {\cite{lovett2018proof}}]\label{lem21} 
Suppose $\mathcal{V}=\{ \mathbf{v}_1,\dots,\mathbf{v}_m \} \subseteq \{0,1\}^{n-2} \times \mathbb{N}^2$ is a minimal counterexample to Theorem~\ref{thml2}. 
Then there do not exist distinct $i,j\in [m]$ such that $\mathbf{v}_i\leqslant \mathbf{v}_j$.
\end{lem}

\begin{lem}[See Lemma 2.2 in {\cite{lovett2018proof}}]
\label{lem22} 
Suppose $\mathcal{V}=\{ \mathbf{v}_1,\dots,\mathbf{v}_m \} \subseteq \{0,1\}^{n-2} \times \mathbb{N}^2$ is a minimal counterexample to Theorem~\ref{thml2}. 
Then $\lvert\mu_{\mathcal{V}}([m])\rvert=0$.
\end{lem}

Suppose $\mathcal{V}$ satisfies $V(k)$. 
A subset $I\subseteq [m]$ is called \textit{tight} for $\mathcal{V}$ if $\sum_{i\in I}(k-|\mathbf{v}_i|)+|\mu_{\mathcal{V}}(I)|=k$, that is, we have equality in (II).

\begin{lem}[See Lemma 2.4 in {\cite{lovett2018proof}}]\label{lem24}
Suppose $\mathcal{V}=\{ \mathbf{v}_1,\dots,\mathbf{v}_m \} \subseteq \{0,1\}^{n-2} \times \mathbb{N}^2$ is a minimal counterexample to Theorem~\ref{thml2}. 
If $I \subseteq [m]$ is tight for $\mathcal{V}$, then $|I|=1$ or $|I|=m$.
\end{lem}

\begin{lem}\label{lemn31}
Let $k,m\geqslant 1$ and $n\geqslant 2$ be integers. Suppose $\mathcal{V}=\{ \mathbf{v}_1,\dots,\mathbf{v}_m \} \allowbreak \subseteq \{0,1\}^{n-2} \times \mathbb{N}^2$ is a minimal counterexample to Theorem~\ref{thml2}.
Then, for some $\alpha,\beta\in \mathbb{N}$, the set $\mathcal{V}$ contains the vectors
\[
  (\underbrace{1,\dots,1}_{n-2},\alpha,0) \quad \text{ and } \quad(\underbrace{1,\dots,1}_{n-2},0,\beta).
\]
Moreover, they are unique in $\mathcal{V}$ with respect to having $n$-th or $(n-1)$-th entry equal to $0$.
\end{lem}

\begin{proof}
First assume $n=2$. By Lemma~\ref{lem22}, we know that for some $\alpha,\beta\in \mathbb{N}$, the vectors $(\alpha,0)$ and $(0,\beta)$ are in $\mathcal{V}$. Furthermore, by Lemma~\ref{lem21}, if for some $\alpha'\in \mathbb{N}$ the vector $(\alpha',0)$ is in $\mathcal{V}$ then $(\alpha',0)=(\alpha,0)$. 
Similarly, if for some $\beta'\in \mathbb{N}$ the vector $(0,\beta')$ is in $\mathcal{V}$ then $(0,\beta')=(0,\beta)$.

Now assume $n\geqslant 3$. 
By Lemma~\ref{lem22}, we know that there exists $i' \in [m]$ such that $\mathbf{v}_{i'}(n)=0$. 
We will show that $\mathbf{v}_{i'}=(1,\dots,1,\alpha,0)$ for some $\alpha\in \mathbb{N}$. 
Suppose (for a contradiction) that there exists $j'\in [n-2]$ such that $\mathbf{v}_{i'}(j')=0$. 
Without loss of generality, assume that $i'=m$ and $j'=n-2$. Let us define a new set of vectors $\mathcal{V}'=\{ \mathbf{v}_1',\dots,\mathbf{v}_m' \}\subseteq \mathbb{N}^{n-1}$ where
$$
\mathbf{v}_i'=(\mathbf{v}_i(1),\dots,\mathbf{v}_i(n-3),\mathbf{v}_i(n-2)+\mathbf{v}_i(n),\mathbf{v}_i(n-1))
$$
for all $i\in [m]$. 
It is clear that $\mathcal{V}'$ has properties (I) and (III). 
To prove that $\mathcal{V}'$ satisfies (II), we use the same steps as in the proof of Lemma 2.5 in \cite{lovett2018proof}.
For completeness we include these steps below.

Let $I \subseteq [m]$.
Clearly (II) holds if $|I|=1$, so suppose $|I|>1$.
We have
\begin{align}
\sum_{i\in I}(k-|\mathbf{v}_i'|)+|\mu_{\mathcal{V}'}(I)|=\sum_{i\in I}(k-|\mathbf{v}_i|)+|\mu_{\mathcal{V}}(I)|+\delta, \label{eqn:3}
\end{align}
where $\delta=\text{min}_{i\in I}(\mathbf{v}_i(n-2)+\mathbf{v}_i(n))-\text{min}_{i\in I}\mathbf{v}_i(n-2)-\text{min}_{i\in I}\mathbf{v}_i(n)$.
Suppose $|I|<m$ so, by Lemma~\ref{lem24}, the subset $I$ is not tight for $\mathcal{V}$.
Thus
$$
\sum_{i\in I}(k-|\mathbf{v}_i|)+|\mu_{\mathcal{V}}(I)|\leqslant k-1.
$$
Note that $\mathbf{v}_i(n-2)\in \{0,1\}$ for all $i\in I$ since $\mathcal{V}$ satisfies (III).
It follows that $\text{min}_{i\in I}(\mathbf{v}_i(n-2)+\mathbf{v}_i(n))\leqslant \text{min}_{i\in I}(1+\mathbf{v}_i(n))=1+\text{min}_{i\in I}\mathbf{v}_i(n)$.
Hence $\delta \leqslant 1-\text{min}_{i\in I}\mathbf{v}_i(n-2) \leqslant 1$.
From $\eqref{eqn:3}$ we obtain
$$
\sum_{i\in I}(k-|\mathbf{v}_i'|)+|\mu_{\mathcal{V}'}(I)|\leqslant k-1+\delta \leqslant k.
$$
Now suppose $|I|=m$.
Here we have $\mathbf{v}_m(n-2)=\mathbf{v}_m(n)=0$.
This implies that $\delta=0$ and, from \eqref{eqn:3}, we find that
$$
\sum_{i\in I}(k-|\mathbf{v}_i'|)+|\mu_{\mathcal{V}'}(I)|=\sum_{i\in I}(k-|\mathbf{v}_i|)+|\mu_{\mathcal{V}}(I)|\leqslant k.
$$
In any case, the set $\mathcal{V}'$ satisfies (II).
Therefore $\mathcal{V}'$ satisfies $V_2(k)$ and $\mathcal{V}'$ has parameters $(n-1,k,m,d)$.

Observe that each polynomial in $P(k,\mathcal{V}')$ can be obtained from a polynomial in $P(k,\mathcal{V})$ by substituting $x_{n-2}$ for $x_n$. 
This operation preserves linear dependence and hence, since $P(k,\mathcal{V})$ is linearly dependent over $\mathbb{F}_q(x_1,\dots,x_n)$, so too is $P(k,\mathcal{V}')$. 
But this contradicts the minimality of $\mathcal{V}$. Therefore, $(1,\dots,1,\alpha,0)$ belongs to $\mathcal{V}$ for some $\alpha\in \mathbb{N}$. 

Furthermore, suppose that there exists a vector $\mathbf{w}$ in $\mathcal{V}$ such that $\mathbf{w}(n)=0$. 
Using the same argument as above, we can conclude that $\mathbf{w}=(1,\dots,1,\alpha',0)$ for some $\alpha'\in \mathbb{N}$. 
By Lemma~\ref{lem21}, we must have $\mathbf{w}=(1,\dots,1,\alpha,0)$. 
In short, the vector in $\mathcal{V}$ where its last coordinate is zero is unique and it takes the form of $(1,\dots,1,\alpha,0)$ for some $\alpha\in \mathbb{N}$. 

Similarly, the vector $(1,\dots,1,0,\beta)$ is in $\mathcal{V}$ for some $\beta\in \mathbb{N}$ and there is only one vector in $\mathcal{V}$ where its second to last coordinate is zero.
\end{proof}

\begin{lem}\label{lemn32}
Let $k,m\geqslant 1$ and $n\geqslant 2$ be integers. 
Suppose $\mathcal{V}=\{ \mathbf{v}_1,\dots,\mathbf{v}_m \} \allowbreak \subseteq \{0,1\}^{n-2} \times \mathbb{N}^2$ is a minimal counterexample to Theorem~\ref{thml2}.
Then the set $\mathcal V$ contains the vectors
\[
  (\underbrace{1,\dots,1}_{n-2},k-n+1,0) \quad \text{ and } \quad (\underbrace{1,\dots,1}_{n-2},0,k-n+1).
\]
\end{lem}

\begin{proof}
By Lemma~\ref{lemn31}, we know that $(1,\dots,1,\alpha,0)\in \mathcal{V}$ for some $\alpha\in\mathbb{N}$. 
Without loss of generality, assume that $\mathbf{v}_m=(1,\dots,1,\alpha,0)$. 
By (I), we have $n-2+\alpha\leqslant k-1$. 
Assume towards a contradiction that $n-2+\alpha\leqslant k-2$. 
Let us define a new set of vectors $\mathcal{V}'=\{ \mathbf{v}_1',\dots,\mathbf{v}_m' \}\subseteq \mathbb{N}^n$ where
$$
\mathbf{v}_1'=\mathbf{v}_1,\dots,\mathbf{v}_{m-1}'=\mathbf{v}_{m-1},\mathbf{v}_m'=(1,\dots,1,\alpha,1).
$$
Note that if $n-2+\alpha\leqslant k-2$ then $k-\alpha-n\geqslant 0$. 
Hence $|P(k,\mathbf{v}_m')|=k-(n-1+\alpha)=k-n-\alpha+1\geqslant 1$ so there is at least one polynomial in $P(k,\mathbf{v}_m')$. 
Clearly $\mathcal{V}'$ satisfies (III) and, by our assumption, $|\mathbf{v}_m'|=n-1+\alpha\leqslant k-1$.
Hence $\mathcal{V}'$ also satisfies (I).

Now consider a subset $I\subseteq [m]$. 
If $m\notin I$ then $\mathcal{V}'$ clearly satisfies (II).
Otherwise, assume that $m\in I$.
Since $\mathbf{v}_m=(1,\dots,1,\alpha,0)$ is the unique vector in $\mathcal{V}$ with last coordinate zero, we have $|\mu_{\mathcal{V}'}(I)|=|\mu_{\mathcal{V}}(I)|+1$. 
We also have $|P(k,\mathbf{v}_m')|=|P(k,\mathbf{v}_m)|-1$ and thus
$$
\sum_{i\in I}(k-|\mathbf{v}_i'|)+|\mu_{\mathcal{V}'}(I)|=\left(\sum_{i\in I}(k-|\mathbf{v}_i|)-1\right)+(|\mu_{\mathcal{V}}(I)|+1) \leqslant k.
$$
Hence $\mathcal{V}'$ satisfies $V_2(k)$.

The set $\mathcal{V}'$ has parameters $(n,k,m,d-1)$. By minimality of $\mathcal{V}$, the polynomials in $P(k,\mathcal{V}')$ are linearly independent over $\mathbb{F}_q(x_1,\dots,x_n)$. 
If we let $r=(x-x_{n-1})^{\alpha}\prod_{j\in [n-2]}(x-x_j)$ then $P(k,\mathcal{V})$ and $P(k,\mathcal{V}')\cup \{r\}$ span the same linear space of polynomials over $\mathbb{F}_q(x_1,\dots,x_n)$. 
It follows that $P(k,\mathcal{V}')\cup \{r\}$ is linearly dependent over $\mathbb{F}_q(x_1,\dots,x_n)$ while $P(k,\mathcal{V}')$ is linearly independent over $\mathbb{F}_q(x_1,\dots,x_n)$. 

Hence we can write $r$ as a linear combination of polynomials in $P(k,\mathcal{V}')$ over $\mathbb{F}_q(x_1,\dots,x_n)$. 
However, the uniqueness of $\mathbf{v}_m=(1,\dots,1,\alpha,0)$ in $\mathcal{V}$ implies that any polynomial in $P(k,\mathcal{V}')$ is divisible by $(x-x_n)$. 
Thus we obtain a contradiction since $(x-x_n)$ does not divide $r$. 
Therefore, $n-2+\alpha=k-1$, which means that $(1,\dots,1,\alpha,0)=(1,\dots,1,k-n+1,0)$. 

Adopting the same method as above, we can also obtain $(1,\dots,1,0,\beta)=(1,\dots,1,0,k-n+1)$. 
Therefore, we have $\alpha=\beta=k-n+1$.
\end{proof}

Now we are ready to prove Theorem~\ref{thml2}.

\begin{proof}[Proof of Theorem~\ref{thml2}]
Suppose that there is a minimal counterexample $\mathcal{V}=\{ \mathbf{v}_1,\dots,\mathbf{v}_m \} \subseteq \{0,1\}^{n-2} \times \mathbb{N}^2$. 
We will derive a contradiction to $\mathcal{V}$ being a counterexample.

By Lemma~\ref{lemn32}, the vector $(1,\dots,1,k-n+1,0)$ is in $\mathcal{V}$. 
Assume without loss of generality that $\mathbf{v}_m=(1,\dots,1,k-n+1,0)$. 
Let $\mathcal{V}'=\{\mathbf{v}_1,\dots,\mathbf{v}_{m-1}\}$ so $\mathcal{V}'$ still satisfies $V_2(k)$. By minimality of $\mathcal{V}$, the set $P(k,\mathcal{V}')$ is linearly independent over $\mathbb{F}_q(x_1,\dots,x_n)$. 
Moreover, since $|\mathbf{v}_m|=k-1$, we have $P(k,\mathcal{V})=P(k,\mathcal{V}')\cup\{r\}$ where $r=(x-x_{n-1})^{k-n+1}\prod_{j\in [n-2]}(x-x_j)$. 
Any polynomial in $P(k,\mathcal{V}')$ is divisible by $(x-x_n)$ while $r$ is not. 
Since $P(k,\mathcal{V})$ is linearly dependent, we can write $r$ as a linear combination of polynomials in $P(k,\mathcal{V}')$ over $\mathbb{F}_q(x_1,\dots,x_n)$, which contradicts that $(x-x_n)$ does not divide $r$.
\end{proof}

Note that, for any $n\geqslant 1$, if $\mathcal{V}\subseteq \mathbb{N}^n$ satisfies $V_0(k)$ then $\mathcal{V}$ satisfies $V_1(k)$.
Combining this with Theorem~\ref{thml1} and Theorem~\ref{thml2}, we can state Theorem~\ref{thml012} below.

\begin{theorem}\label{thml012}
Let $k,m,n\geqslant 1$ and $l\geqslant 0$ be integers where $n\geqslant l$ and let $\mathcal{V}=\{ \mathbf{v}_1,\dots,\mathbf{v}_m \} \subseteq \mathbb{N}^n$. Assume that $l\leqslant 2$ and $\mathcal{V}$ satisfies $V_l(k)$. Then the polynomials in $P(k,\mathcal{V})$ are linearly independent over $\mathbb{F}_q(x_1,\dots,x_n)$.
\end{theorem}

\section*{Acknowledgments}

The authors have benefited from conversations with Han Mao Kiah.


\bibliographystyle{abbrvnat}
\bibliography{references}

\end{document}